\documentclass{amsart}
\usepackage{amsmath, amsthm, amssymb, amsfonts}
\usepackage{bm}
\usepackage{dsfont}
\usepackage[utf8]{inputenc}
\usepackage{rotate}
\usepackage{tikz}
\usepackage{tikz-cd}
\usepackage[arrow,matrix,curve]{xy} 	
\usepackage{xcolor}
\usepackage{lipsum}
\usepackage{caption}
\usepackage{subcaption}
\usepackage{todonotes}

\theoremstyle{plain}
\newtheorem{theorem}{Theorem}[section]

\theoremstyle{definition}

\newtheorem{remark}[theorem]{Remark}

\newcommand{\norm}[1]{\left\lVert#1\right\rVert}

\newcommand{\mathset}[2]{\left\{#1\middle\vert #2 \right\}}
\newcommand{\mathseq}[2]{\left(#1\right)_{#2}}

\newcommand{\card}[1]{{\# #1}}
\newcommand{\N}{\mathbb{N}}
\newcommand{\Z}{\mathbb{Z}}

\title{On the Pair Correlation Statistic of Sequences with Finite Gap Property}
\date{\today}

\author{Jasmin Fiedler and Christian Weiß}
\address{{\bf{Ruhr West University of Applied Sciences,}}\\ {{Department of Natural Sciences, Duisburger Str. 100,}}\\{{45479 M\"ulheim an der Ruhr, Germany}}}
\email{jasmin.fiedler@hs-ruhrwest.de, christian.weiss@hs-ruhrwest.de}
\begin{document}

\begin{abstract} The limiting function $f(s)$ of the pair correlation
\[ 
\frac{1}{N} \# \left\{ 1 \leq i\neq j\leq N \middle\vert  \left\lVert x_i - x_j \right\rVert \leq \frac{s}{N} \right\} 
\]
for a sequence $(x_N)_{N \in \mathbb{N}}$ on the torus $\mathbb{T}^1$ is said to be Poissonian if it exists and equals $2s$ for all $s \geq 0$. For instance, independent, uniformly distributed random variables generically have this property. Obviously $f(s)$ is always a monotonic function if existent. There are only few examples of sequences where $f(s) \neq 2s$, but where the limit can still be explicitly calculated. Therefore, it is an open question which types of functions $f(s)$ can or cannot appear here. In this note, we give a partial answer on this question by addressing the case that the number of different gap lengths in the sequence is finite and showing that $f$ cannot be continuous then.
\end{abstract}

\maketitle

\section{Introduction}\label{sec:intro}

For a sequence $\mathseq{x_N}{N\in\N}$ on the torus $\mathbb{T}^1 = [0,1] / \sim$, a natural object to study is the behavior of gaps on a local scale, i.e., 
\[
\frac{1}{N} R(s,N):= \frac{1}{N} \card{\mathset{1 \leq i\neq j\leq N}{\norm{x_i-x_j}\leq\frac{s}{N}}}
\]
for $s \geq 0$, where $\norm{x}:=\min\mathset{|x-z|}{z\in\Z}$ represents the distance to the nearest integer. If $\mathseq{X_N}{N\in\N}$ is a sequence of independent and uniformly distributed random variables it can easily be shown that
$$\lim_{N \to \infty} \frac{1}{N} R(s,N) = 2s$$
holds almost surely for all $s \geq 0$, see e.g. \cite{Mar07}. Whenever a deterministic sequence $\mathseq{x_N}{N\in\N}$ has this property, it has \textit{Poissonian pair correlations}. The first explicit example of such a sequence was found in \cite{BMV15} and is given as $(x_N) = \{ \sqrt{N} \}$ for $N$ not a perfect square, where $\{ \cdot \}$ denotes the fractional part of a number. Other more recent examples are due to \cite{LST21} and \cite{LT22}. The number of examples is so limited because calculating the pair correlation is usually a hard task. On the contrary, many important classes of uniformly distributed sequences, like van der Corput and Kronecker sequences, can even be excluded to have Poissonian pair correlations, see e.g. \cite{larcher:som_neg_results_poiss_pair_corr}, \cite{PS19}.\\[12pt]
One general obstacle for possessing Poissonian pair correlations lies in the gap structure of the sequence. Recall that for a finite sequence $(x_n)_{n=1}^N \in \mathbb{T}^1$ the gap lengths are the distances between two neighboring elements. Throughout this note, the different gap lengths are denoted by $d^1,\ldots,d^k$ with $k \leq N$. For instance, Kronecker sequences $\{ n \alpha\}$ with $\alpha \in \mathbb{R}$ only have at most three different gap lengths according to the famous Three Gap Theorem, see \cite{Sos58}. The same property holds true for an der Corput sequences.
\begin{theorem}{(\cite{larcher:som_neg_results_poiss_pair_corr}, Theorem 1)} \label{thm:LS} Let $\mathseq{x_N}{N\in\N}$ be a sequence in $\mathbb{T}^1$ with the following property: There is an $s \in \mathbb{N}$, positive real numbers $K,\gamma \in \mathbb{R}$, and infinitely many $N \in \mathbb{N}$ such that the point set $x_1,\ldots,x_N$ has a subset with $M \geq \gamma N$ elements, denoted by $x_{j_1},\ldots,x_{j_M}$, which are contained in a set of points with cardinality at most $KN$ having at most $s$ gap lengths. Then $\mathseq{x_N}{N\in\N}$ does not have Poissonian pair correlations.
\end{theorem}
Further connections between the gap properties of sequences and the pair correlation statistic have been established e.g. in \cite{AISTLEITNER2021112555} and \cite{weiss:some_conn_discr_fin_gap_pair_corr}. It is now natural to ask which other types of limits may occur for the pair correlation statistic, i.e., given a function $f: \mathbb{R}^+_0 \to \mathbb{R}^+_0$, does there exist a sequence $\mathseq{x_N}{N\in\N} \in \mathbb{T}^1$ such that
\begin{align} \label{eq:fppc}
\lim_{N \to \infty} \frac{1}{N} R(N,s) = f(s).
\end{align}
for all $s \geq 0$? We say that any sequence $\mathseq{x_N}{N\in\N} \subset [0,1]$ which satisfies \eqref{eq:fppc} has \textit{$f$-pair correlations}. Also in the non-Poissonian case, there is only a limited number of known examples, i.e. functions $f$, such that there exists a sequence with $f$-pair correlations, see \cite{Lut20} \cite{MS13}, \cite{Say23}, \cite{Wei23}, \cite{Wei24}. It is trivial that $f$ must be non-decreasing. Moreover $f(0) = 0$ is automatic if the number of multiply appearing elements is of order $o(\sqrt{N})$. This happens for example if the elements of $\mathseq{x_N}{N\in\N}$ are distinct. To the best of the authors' knowledge there is nothing else known about any property which a general $f$ necessarily has or does not have. In this note, we address the case, where $f$ only has a finite number of different gap lengths and show the following:
\begin{theorem} \label{thm:main}
Let $f:\mathbb{R}^+_0\to \mathbb{R}^+_0$ be a continuous and non-decreasing function. Let $x=\mathseq{x_N}{N\in\N}$ be a sequence in $[0,1)$. If there is some $k\in\N$ and infinitely many $N\in\N$ such that $x_1,\ldots,x_N$ has at most $k$ different distinct gap lengths, then $x$ does not have $f$-pair correlations. 
\end{theorem}
\begin{remark} The careful reader will realize, that our proof of Theorem~\ref{thm:main} is also valid if $f$ is left-continuous on $\mathbb{R}^+$, continuous in $0$ and has only a finite number of discontinuities. 
\end{remark}
Note that Theorem~\ref{thm:main} leaves two possibilities: either the limit does not exist or it is not continuous. For example, it is easy to see for a Kronecker sequence $\{n \alpha\}$ with $\alpha \in \mathbb{R} \setminus \mathbb{Q}$ that the limit cannot exist. The same holds true for van der Corput sequences, compare e.g. \cite{Wei24}.\\[12pt]
The strategy of our proof resembles the simplified proof of Theorem~\ref{thm:LS} in \cite{FW23} with a few additional arguments to account for the stronger claim. The arguments therein rely on a classification of the gap lengths depending on how fast they approach zero as $N$ grows. Therefore, the proof crucially depends on the finite gap property.\\[12pt]
Finally, it is already clear from the Poissonian case that sequences with $f$-pair correlations for continuous $f$ exist. However, we would like to stress that there are also other known examples of continuous $f$-pair correlations, see e.g. \cite{Wei23}, so the finite gap property is indeed necessary in Theorem~\ref{thm:main}, also when we only consider the non-Poissonian situation.

\paragraph{Acknowledgments.} The authors thank Michael Gnewuch for his valuable comments on an earlier draft of this paper and pointing out some inaccuracies to us.

\section{Proof of the main result}

In this section we prove Theorem~\ref{thm:main}.

\begin{proof}[Proof of Theorem~\ref{thm:main}]
The sequence $x=\mathseq{x_N}{N\in\N}$ contains a subsequence such that it has exactly $k$ distinct distances between neighboring elements (pigeonhole principle). We will without loss of generality assume that $x$ is that sequence. Let us call these distances $d_1^{N},\ldots,d_k^{N}$ and order them such that $d_1^N<d_2^n<\ldots<d_k^N$ for all $n\in\N$. \\[12pt]
Without loss of generality we may also assume that the set $x_{1},\ldots,x_{N}$ is ordered, meaning that gaps are always between consecutive elements.\\[12pt]
By moving to subsequences if necessary, every gap length $\mathseq{d_j^{N}}{N\in\N}$ with $1\leq j\leq k$ falls into one of three categories (compare \cite{larcher:som_neg_results_poiss_pair_corr} and \cite{FW23}):
\begin{itemize}
    \item $\lim_{N\to\infty}Nd_j^{N}=\infty$, which we call large gaps.
    \item $d_j^{N}=0$ for all $i\in\N$. We call these gaps zero gaps.
    \item There does not exist a subsequence of large gaps and no $d_j^{N}$ equals zero. In this case, we speak of medium gaps. Note that this implies that there are constants $0<U_j<\infty$ such that $0< {N} d_j^{N}\leq U_j$ for all ${N}\in\N$.
\end{itemize}
From now on, we assume that the sequence has $f$-pair correlations, aiming for a contradiction. For every $1\leq j\leq k$, let 
\[
l_{j}^{\N}=\card{\mathset{i\leq {N} \, }{ \, \Vert x_{i}-x_{i+1}\Vert=d_{j}^{N}}}
\]
where we define $x_{{N}+1}:=x_1$. By definition
\begin{equation}\label{eq:sum_vec_n_eq_N_i}
\sum_{j=1}^kl_{j}^{N}={N}
\end{equation}
is true for all ${N}\in\N$. We will now show that for every $j$ that is not a zero gap $\lim_{{N}\to\infty}\frac{l_{j}^{N}}{{N}}=0$ holds, i.e. those gaps contribute to the sum in \eqref{eq:sum_vec_n_eq_N_i} in a negligible way. Afterwards we will use this and \eqref{eq:sum_vec_n_eq_N_i} to arrive at a contradiction in the case when $d_j^{N}$ represents  zero gaps.\\[12pt]
First, we take a look at the case, where the sequence contains s a large gap of size $d_j^{N}$. Interestingly, for the $f$-pair correlation property this case is irrelevant. Notice that 
\[
l_{j}^{N}\leq\frac{1}{d_j^{N}}.
\]
It follows that
\[
    0 \leq \limsup_{{N}\to\infty}\frac{l_{j}^{N}}{{N}}\leq \limsup_{{N}\to\infty}\frac{1}{d_j^{N}{N}}=0.
\]
Now, let us take a look at medium gaps.  We will once again show that $\lim_{{N}\to\infty}\frac{l_{j}^{N}}{{N}}=0$.
Let $1\leq j\leq k$ correspond to a medium gap and assume to the contrary that there is some $\delta>0$ such that $ l_{j}^{N}\geq \delta {N}$ for infinitely many ${N} \in \mathbb{N}$. 
 Since the gaps are medium, there is a constant $0<U_j<\infty$ such that $0< d_{j}^{N} {N}\leq U_j$ for all ${N}\in\N$. Since $f$ is continuous, the interval $[0,U_j]$ can be split into a finite number of subintervals such that over the course of each interval $f$ increases by less than $\delta$. By the pigeonhole principle one of these intervals must contain an infinite number of $d_j^{N} N$. This means that there exist $0\leq u<t\leq U_j$ with $f(t)-f(u)<\delta$ and 
\begin{equation}\label{eq:u_low_d_leq_t}
 \frac{u}{{N}}< {d}_{j}^{N}\leq\frac{t}{{N}}   
\end{equation}
for infinitely many ${N}\in\N$. After moving to a subsequence, if necessary, this holds true for all ${N}$.  \\[12pt]
For every $i\leq {N}$ with $\Vert x_{i}-x_{i+1}\Vert=d_{j}^{N}$ inequality \eqref{eq:u_low_d_leq_t} implies $$\frac{u}{{N}} < \Vert x_{i}-x_{i+1}\Vert\leq \frac{t}{{N}}.$$ 
Thus,
\begin{align*}
    R(t,{N})
    &\geq R(u,{N})+2l_{j}^{N}.
\end{align*}
As the sequence has $f$-pair correlations by assumption, this implies 
\begin{align*}
    f(t)-f(u)
    &=\lim_{{N}\to\infty} \frac{1}{{N}}R\left(t, {N}\right)-\frac{1}{{N}}R(u, {N})\\
    &\geq\limsup_{{N}\to\infty}\frac{2l_{j}^N}{N}\geq 2\delta.
\end{align*}
But we chose our points in such a way that $
f(t)-f(u)<\delta$, which is a contradiction. Therefore,  $\lim_{i\to\infty}\frac{l_j^{N}}{{N}}=0$.\\[12pt]
Finally, we consider zero gaps. If there were none, then 
\[
\lim_{i\to\infty}\frac{\sum_{j=1}^k l_j^{N}}{{N}}=0,
\]
by the cases of medium and large gaps which we just considered. This contradicts \eqref{eq:sum_vec_n_eq_N_i}. Hence, there must be zero gaps. Since we ordered the gap lengths $d_1^{N}<d_2^{N}\ldots<d_k^{N}$ by size, only $d_1^{N}$ represents a zero gap and again by \eqref{eq:sum_vec_n_eq_N_i} we conclude that $\lim_{{N}\to\infty}\frac{ l_1^{N}}{{N}}=1$. \\
Setting $l^{N}:=\sum_{j=2}^kl_j^{N}$ it follows that $l^{N}=o_{N}({N})$. The sequences consists of $l^{N}$ batches of points which have distance $0$ each. These batches are inductively given by $B_1:=\mathset{i\leq {N}}{x_{i}=x_1}$, $B_{i}:=\mathset{i_0\leq {N}}{x_{i_0}=x_{\nu_l}}$ with $\nu_l=\max B_{i-1}+1$ for $i=2, \ldots, l^{N}$. By construction the batches contain all points, i.e. $\sum_{{i}=1}^{l^{N}}\card{B_{i}}={N}$. Hence, we obtain
\begin{align*} 
  \frac{1}{{N}}R(s,{N})&\geq\frac{1}{{N}}\sum_{{i}=1}^{l^{N}}\card{\mathset{i_0\neq j_0\in B_{i}}{\norm{X_{i_0}-X_{j_0}}\leq \frac{s}{{N}}}}\\
  &=\frac{1}{{N}}\sum_{{i}=1}^{l^{N}}\card{B_{i}}(\card{B_{i}}-1)\\
  &=\left(\frac{1}{{N}}\sum_{{i}=1}^{l^{N}}\card{B_{i}}^2\right)-1
  \end{align*}
  The sum on the right hand side is minimal if all the batches $B_{i}$ have the same cardinality leading to
\begin{align*} 
  \frac{1}{{N}}R(s,{N})&\geq\frac{1}{{N}}\sum_{{i}i=1}^{l^{N}}\card{B_{i}}^2-1\\
  &\geq \frac{1}{{N}}\sum_{{i}=1}^{l^{N}}\left(\frac{{N}}{l^{N}}\right)^2-1\\
  &=\frac{{N}}{l^{N}}-1\stackrel{{N}\to\infty}{\longrightarrow}\infty
  \end{align*}
since $l^{N}=o({N})$. This contradicts the assumption that $\mathseq{x_N}{N\in\N}$ has $f$-pair correlations since the limit of $\frac{1}{{N}}R(s,{N})$ cannot even be finite and thus finishes the proof.

\end{proof}

\bibliographystyle{alpha}
\bibdata{references}
\bibliography{references}
\end{document}